\documentclass[10pt]{article}
\usepackage{graphicx}
\usepackage[a4paper, textwidth=16cm]{geometry}

\usepackage{amsthm,amsmath,mathtools}
\usepackage{xcolor,paralist,hyperref,titlesec,fancyhdr,etoolbox}
\newtheorem{theorem}{Theorem}[]
\newtheorem{definition}[theorem]{Definition}
\newtheorem{lemma}[theorem]{Lemma}

\usepackage{amssymb}
\newtheorem{assumption}{Assumption}

\titleformat{\section}
  {\normalfont\Large\bfseries}
  {\thesection}
  {1em}
  {}
  
\titlespacing*{\section}{0pt}{0ex}{0ex}

\hypersetup{ colorlinks=true, linkcolor=black, filecolor=black, urlcolor=black, citecolor=black }

\usepackage{lipsum}

\begin{document}
\title{\textbf{Stable boundary determination of a complex anisotropic admittivity and its derivatives from a local Neumann-to-Dirichlet map}} 
\author{Jessica Crosse\thanks{Department of Mathematics and Statistics, University of Limerick, V94 T9PX, Limerick, Ireland. E-mail: jessica.crosse@ul.ie}
\and Romina Gaburro\thanks{Department of Mathematics and Statistics, University of Limerick, Health Research Institute (HRI), V94 T9PX, Limerick, Ireland. E-mail: romina.gaburro@ul.ie}}
\date{}
\maketitle

\begin{abstract}
We study the classical anisotropic Calderón problem associated to the elliptic equation $\text{div}(\sigma\nabla u)=0$, where the complex admittivity $\sigma$ is of the form $\sigma=A(\cdot,a(\cdot))$ in a domain $\Omega\subset\mathbb{R}^n$, $n\ge3$. We establish boundary stability estimates for $\sigma$ and its derivatives of arbitrary order from a local Neumann-to-Dirichlet map. Our results extend those of Comm. Partial Differential Equations, 34 (2009) from the real-valued conductivity setting to complex anisotropic admittivities. 
\end{abstract} 


\section{Introduction}
\label{sec:intro}
In this paper, we address the inverse problem of stably determining the complex anisotropic admittivity $\sigma$ in a domain $\Omega\subset\mathbb{R}^n$, with $n\geq3$, from the Neumann-to-Dirichlet (N-D) map localised on a portion $\Sigma$ of the boundary of $\Omega$, $\partial\Omega$. In absence of internal sources, the electrostatic potential $u$ in $\Omega$ satisfies
\begin{equation}
    \label{eqn: 1}
    Lu=\textrm{div}(\sigma\nabla u)=0 ,\quad\textrm{in}\quad\Omega,
\end{equation}
where the (possibly anisotropic) electric admittivity at frequency $k$ is given by the complex-symmetric matrix valued function $\sigma(x)=\sigma_R(x) + ik\sigma_I(x)$, $x\in\Omega$,
where $\sigma_R$ and $\sigma_I$ are the \textit{conductivity} and \textit {permittivity} of $\Omega$, respectively. 

When $\sigma$ is real and of type $\sigma = A(\cdot,a(\cdot))$, where $t\mapsto A(\cdot, t)$ is \textit{a-priori} known and $a$ is an unknown scalar parameter to be determined, Lipschitz and H\"older stability estimates for $\sigma$ and its derivatives of any order, respectively, were established in \cite{A-G2} via the method of singular solutions to \eqref{eqn: 1} introduced in \cite{A}. Such solutions had an isolated singularity of order $2-n-m$, for any positive integer $m$ at a point outside $\Omega$. The complex case has been less studied to date. When $\sigma$ is complex and again of type $\sigma = A(\cdot,a(\cdot))$ and the boundary map is the local Dirichlet-to-Neumann (D-N) map, similar results were recently obtained in \cite{mypreprint}. In the complex setting, we also recall the global stability result \cite{FosGabSin25}, which treats the case where $\sigma=a A$, where $A$ is a Lipschitz real matrix-valued function which is \textit{a-priori} known and $a$ is a complex affine scalar function to be stably determined in terms of a local D-N map. 

Here we extend the results of \cite{A-G2} to the complex setting, considering complex-matrix valued-functions of the form $\sigma=A(\cdot,a(\cdot))$, where $t\mapsto A(\cdot, t)$ is \textit{a-priori} known and $a$ is an unknown scalar parameter. The precise assumptions are given in Section \ref{sec2}. We also further extend the boundary stability results in \cite{mypreprint}, established in terms of the global D-N map, to the physically relevant setting in which only the local N-D map is available. This choice is motivated by applications to Electrical Impedance Tomography (EIT) where only partial boundary measurements are often accessible via a local N-D map, particularly in medical imaging and geophysical prospecting. We also emphasise that allowing $\sigma$ in \eqref{eqn: 1} to be a complex-valued matrix yields a more realistic model of a material's electrical behaviour by simultaneously accounting for both its conductivity $\sigma_R$ and permittivity $\sigma_I$. We refer to \cite{FosGabSin25} for a discussion on this.

Our stability estimates of $\sigma$ and its derivatives, obtained when the frequency $k$ belongs to an interval explicitly depending on the \textit{a-priori} information, rely on the construction of singular solutions to \eqref{eqn: 1} having an isolated singularity outside $\overline{\Omega}$ of order $2-n-m$, for any integer $m\geq0$, satisfying a zero Neumann conditions on $\partial\Omega\backslash\overline{\Sigma}$. For the case $m=0$, we construct a modified Neumann function of $L$ under the mild assumption that $A$ is H\"older continuous (Theorem \ref{Neumann th}). 

The paper is organised as follows. In Section \ref{sec2}, we rigorously formulate the problem, rigorously define the local N-D map and state our main stability result (Theorem \ref{stab results}). Section \ref{sec: sing solns} is devoted to the construction of singular solutions (Theorems \ref{Neumann th}, \ref{sing thm}). The proof of Theorem \ref{stab results} is presented in Section \ref{sec4}. 


\section{Formulation of the problem and main result}\label{sec2}


For $n\geq3$, a point $x\in\mathbb{R}^n$ will be denoted by $x=(x',x_n)$, where $x'\in\mathbb{R}^{n-1}$ and $x_n\in\mathbb{R}$. Given a point $x\in\mathbb{R}^n$, we will denote with $B_r(x)$, $B'_r(x')$ the open balls in $\mathbb{R}^n$, $\mathbb{R}^{n-1}$, centred at $x$ and $x'$, respectively, with radius $r$.
For any $v,w\in\mathbb{C}^n$, with $v=(v_1,...,v_n),\;w=(w_1,...,w_n)$, we understand that $v\cdot w=\sum_{i=1}^n v_iw_i$. We consider a bounded domain $\Omega\subset\mathbb{R}^n$, with $n\geq 3$, having a $C^1$ boundary $\partial\Omega$ with constant $r_0$, i.e., for any $P\in\partial\Omega$, there exists a rigid transformation of coordinates under which we have $P=0$ and $\Omega\cap B_{r_0}(0)=\{(x',x_n)\in B_{r_0}(0)\;|\;x_n>\varphi(x')\}$, where $\varphi$ is a $C^1$ function on $B'_{r_0}(0)$. Given a constant $\lambda>0$, we consider, for a fixed frequency $k>0$, the one-parameter family of complex matrix-valued functions $A(x,t)=A_R(x,t)+ik A_I(x,t)$, for any $x\in\Omega$, $t\in[\lambda^{-1},\lambda]$. Here $A_R$, $A_I\in L^{\infty}(\Omega\times[\lambda^{-1},\lambda], Sym_n)$ where $Sym_n$ denotes the class of $n\times n$ real-valued symmetric matrices and we assume that $A_R$ and $A_I$ commute. Throughout the entire manuscript we fix a real number $p>n$, denote by $\nu$ the unit outer normal to $\partial\Omega$, and by $\Re(z)$, $\Im(z)$ the real and imaginary part of $z\in\mathbb{C}$, respectively.
\begin{definition}\label{A in H}
    Given positive constants $\mathcal{E}_1,\mathcal{E}_2,E,\mathcal{D}$ and $\lambda$ introduced above, we say that $A(\cdot,\cdot)\in\mathcal{H}$ if the following conditions hold:
    \begin{eqnarray}
         &&A_R,A_I,D_tA_R,D_tA_I\in W^{1,p}(\Omega\times[\lambda^{-1},\lambda],Sym_n),\\
        &&supess_{t\in[\lambda^{-1},\lambda]} \big(\Vert A(\cdot,t)\Vert_{L^p(\Omega)}+\Vert D_xA(\cdot,t)\Vert_{L^p(\Omega)}+\Vert D_tA(\cdot,t)\Vert_{L^p(\Omega)} +\Vert D_tD_xA(\cdot,t)\Vert_{L^p(\Omega)}\big)\leq E,\\
        &&\Re(D_tA(x,t)\xi\cdot\xi)\geq \mathcal{D}^{-1}|\xi|^2,\quad\textrm{for a.e.} \quad x\in\Omega,\quad\textrm{for every}\quad t\in[\lambda^{-1},\lambda],\quad\xi\in\mathbb{C}^n,\label{eqn: mono}\\
        &&\mathcal{E}_1^{-1}|\xi|^2\leq A_R(x,t)\xi\cdot\xi\leq\mathcal{E}_1|\xi|^2,\quad\textrm{for a.e.} \quad x\in\Omega,\quad\textrm{for every}\quad t\in[\lambda^{-1},\lambda],\quad\xi\in\mathbb{R}^n,\label{eqn: AR assump}\\
        &&\begin{aligned}
\mathcal{E}_2^{-1}|\xi|^2
    &\le A_I(x,t)\xi\cdot\xi
    \le \mathcal{E}_2|\xi|^2,\quad\textrm{for a.e.} \quad x\in\Omega,\quad\textrm{for every}\quad t\in[\lambda^{-1},\lambda],\quad\xi\in\mathbb{R}^n,\\
&\text{or}\\
-\mathcal{E}_2|\xi|^2
    &\le A_I(x,t)\xi\cdot\xi
    \le -\mathcal{E}_2^{-1}|\xi|^2,\quad\textrm{for a.e.} \quad x\in\Omega,\quad\textrm{for every}\quad t\in[\lambda^{-1},\lambda],\quad\xi\in\mathbb{R}^n.
\end{aligned}\label{eqn: AI assump}
    \end{eqnarray}
\end{definition}
\begin{assumption}\label{assump on ai} 
Given a positive constant $\mathcal{F}$ and $\lambda$ introduced above, we assume that $a\in L^\infty(\Omega)$ satisfies
    \begin{eqnarray}
      \lambda^{-1}\leq a(x)\leq\lambda,\quad\textrm{for a.e.}\quad x\in\Omega,\quad\textrm{and}\quad\Vert a\Vert_{W^{1,p}(\Omega)}\leq \mathcal{F}.
    \end{eqnarray}
\end{assumption}

By denoting $u=u_1+iu_2$, the complex equation \eqref{eqn: 1} can be written as $\textrm{div}(\mathbf{C}(\cdot,a(\cdot))\nabla \mathbf{u}(\cdot))=0$, in $\Omega$, where $\mathbf{u}=(u_1,u_2)^T$, and $\mathbf{C}\big(\cdot,a(\cdot)\big)=\begin{pmatrix}
A_R\big(\cdot,a(\cdot)\big) & -k A_I\big(\cdot,a(\cdot)\big)\\
k A_I\big(\cdot,a(\cdot)\big) & A_R\big(\cdot,a(\cdot)\big)
\end{pmatrix}$. By \eqref{eqn: AR assump}, there is a constant $C_1>0$ depending on $\mathcal{E}_1$ such that $\mathbf{C}$ satisfies the \textit{strong ellipticity condition} 
\begin{equation}\label{eqn: C elliptic}
    C_1^{-1}|\xi|^2\leq \mathbf{C}(x,a(x))\xi\cdot\xi\leq C_1|\xi|^2,\quad\textrm{for a.e}\quad x\in\Omega,\quad\textrm{for every}\quad \xi\in\mathbb{R}^{2n}.
\end{equation}


We introduce an open portion $\Sigma$ of $\partial\Omega$, set $\Gamma=\partial\Omega\backslash\overline{\Sigma}$ and consider the function spaces
\begin{eqnarray*}
    &&\prescript{}{0}{H}^{\frac{1}{2}}(\partial\Omega)=\big\{f\in H^{\frac{1}{2}}(\partial\Omega)\;\big|\;\int_{\partial\Omega}f=0\big\},\qquad \prescript{}{0}{H}^{-\frac{1}{2}}(\partial\Omega)=\big\{f\in H^{-\frac{1}{2}}(\partial\Omega)\;\big|\;\langle f,1\rangle=0\big\},\\
    &&\prescript{}{0}{H}^{-\frac{1}{2}}(\Sigma)=\big\{\psi\in\prescript{}{0}{H}^{-\frac{1}{2}}(\partial\Omega)\;\big|\;\langle\psi,f\rangle=0\;\textrm{for any}\; f\in H_{00}^{\frac{1}{2}}(\Gamma)\big\},
\end{eqnarray*}
where $H_{00}^{\frac{1}{2}}(\Gamma)$ denotes the closure in the $H^{\frac{1}{2}}(\partial\Omega)$ norm of $\big\{f\in H^\frac{1}{2}(\partial\Omega)\;|\; \textrm{supp} f\subset\Gamma\big\}$.\\
We consider $A\in\mathcal{H}$, $a$ satisfying Assumption \ref{assump on ai} and rigorously define the Neumann-to-Dirichlet map associated to $A(\cdot,a(\cdot))$, together with its local version, localised on $\Sigma$.
\begin{definition}
    The Neumann-to-Dirichlet (N-D) map associated with $A(\cdot,a(\cdot))$, $\mathcal{N}_{A(\cdot, a(\cdot))}:\prescript{}{0}{H}^{-\frac{1}{2}}(\partial\Omega)\rightarrow \prescript{}{0}{H}^{\frac{1}{2}}(\partial\Omega)$
    is given by the selfadjoint operator satisfying $\langle\psi,\overline{\mathcal{N}_{A(\cdot, a(\cdot))}\psi}\rangle=\int_\Omega A(x,a(x))\nabla u(x)\cdot\nabla u(x)\;dx$,
    for every $\psi\in\prescript{}{0}{H}^{-\frac{1}{2}}(\partial\Omega)$, where $u\in H^1(\Omega)$ is the weak solution to the Neumann problem
    \begin{equation*}
    \begin{cases}
        \textrm{div}(A(\cdot,a(\cdot))\nabla u(\cdot))=0,&\quad\textrm{in}\quad\Omega,\\
        A(\cdot,a(\cdot))\nabla u(\cdot)\cdot\nu|_{\partial\Omega}=\psi,& \quad\textrm{on}\quad \partial\Omega,\\
        \int_{\partial\Omega}u=0.
    \end{cases}
\end{equation*}
The local N-D map associated to $A(\cdot,a(\cdot))$ and $\Sigma$ is the operator $\mathcal{N}_{A(\cdot,a(\cdot))}^\Sigma:\prescript{}{0}{H}^{-\frac{1}{2}}(\Sigma)\rightarrow\left(\:\prescript{}{0}{H}^{-\frac{1}{2}}(\Sigma)\right)^\ast\subset\prescript{}{0}{H}^{\frac{1}{2}}(\partial\Omega)$ given by $\langle\mathcal{N}_{A(\cdot,a(\cdot))}^\Sigma\varphi,\overline{\psi}\rangle=\langle\mathcal{N}_{A(\cdot,a(\cdot))}\varphi,\overline{\psi}\rangle$, for every $\varphi,\psi\in\prescript{}{0}{H}^{-\frac{1}{2}}(\Sigma)$.
\end{definition}
To emphasise the dependence of $\mathcal{N}_{A(\cdot,a(\cdot))}^\Sigma$ on $a$ and $\Sigma$, we will simply denote it by $\mathcal{N}_a^\Sigma$. 
The following identity can be recovered from the well-known Alessandrini's identity \cite{A} (see also \cite{AleDeHGabSin18} for the real case and \cite{FosGabSin25} for the complex one),
\begin{equation}\label{eqn: Aless for ND}
    \big\langle A(x,a_1(x))\nabla u_1\cdot\nu,\overline{(\mathcal{N}^\Sigma_{a_2}-\mathcal{N}^\Sigma_{a_1})A(x,a_2(x))\nabla u_2\cdot\nu}\big\rangle=\int_\Omega\big(A(x,a_1(x))-A(x,a_2(x))\big)\nabla u_1(x)\cdot\nabla u_2(x)\;dx,
\end{equation}
for any $u_i\in H^1(\Omega)$ being the unique weak solution to $\textrm{div}(A(\cdot,a_i(\cdot))\nabla u_i)=0$, in $\Omega$, for $i=1,2$. We will denote 
\begin{flalign}\label{eqn: DN norm def}
        \Vert&\mathcal{N}_{a}^\Sigma\Vert_{\ast}=\sup\big\{|\langle f,\overline{\mathcal{N}_{a}^\Sigma g}\rangle|\;|\; f,g\in\prescript{}{0}{H}^{-\frac{1}{2}}(\Sigma),\; \Vert f\Vert_{\prescript{}{0}{H}^{-1/2}(\Sigma)}=\Vert g\Vert_{\prescript{}{0}{H}^{-1/2}(\Sigma)}=1\big\}
    \end{flalign}
to be the norm on the Banach space of bounded linear operators between $\prescript{}{0}{H}^{-\frac{1}{2}}(\Sigma)$ and $(\prescript{}{0}{H}^{-\frac{1}{2}}(\Sigma))^\ast$. 
Let $\partial\Sigma$ denote the boundary of $\Sigma$. For every $\eta$, $0<\eta<r_0$, we denote
    \begin{eqnarray}
        \Sigma_\eta =\big\{x\in\Sigma\;|\,\textrm{dist}(x, \partial\Sigma)>\eta\big\},\quad U_\eta =\Big\{x\in\mathbb{R}^n\;|\;\textrm{dist}(x,\Sigma_\eta)<\frac{\eta}{4}\Big\},\quad\textrm{and}\quad U_\eta^i =U_\eta\cap\Omega.
    \end{eqnarray}
Given $\Sigma$ is open and non-empty, there exists $\eta_0$ such that for $0<\eta_0<r_0$, $\Sigma_{\eta_0}$ is always non-empty. From now on we shall only consider values of $\eta$ below $\eta_0$. We will refer to the set of positive numbers $n,p,k,r_0, \lambda,\mathcal{E}_1, \mathcal{E}_2, E,\mathcal{D},$ $\mathcal{F}, \eta,\eta_0$ introduced above, and the diameter of $\Omega$, $diam(\Omega)$ as the \textit{a-priori} data.


\begin{theorem}[Local stability of boundary values and their derivatives]\label{stab results}
   Assume that $k$ satisfies
    \begin{equation}
        \label{eqn: k1}
        0\!<k\!\leq\!\!\underset{\mathcal{A}+\mathcal{B}+\mathcal{C}=1}{\max}\!\bigg\{\!\min \!\bigg\{\frac{(m^3-m^{-3})\tan(\frac{\mathcal{A}\pi}{4})}{M^3-M^{-3}}\!, \; M^{-6}\tan\bigg(\frac{\mathcal{B}\pi}{2n}\bigg)\!,\; M^{-6}\tan\bigg(\frac{\mathcal{C}\pi}{2n}\bigg)\bigg\}\!\bigg\},
    \end{equation}
   where $M=\max\{\mathcal{E}_1,\mathcal{E}_2\}$, $m=\min\{\mathcal{E}_1,\mathcal{E}_2\}$. If $A\in\mathcal{H}$ and $a_i$ is a real-valued function satisfying Assumption \ref{assump on ai}, for $i=1,2$, then we have
\begin{equation}
    \label{eqn: Stab result}
    \Vert A\big(x,a_1(x)\big)-A\big(x,a_2(x)\big)\Vert_{L^\infty(\overline{\Sigma}_\eta)}\leq C\Vert\mathcal{N}_{a_1}^\Sigma-\mathcal{N}_{a_2}^\Sigma\Vert_\ast,
\end{equation}
where $C$ is a positive constant which depends on the \textit{a-priori} data only. Furthermore, if there exists $E_h>0$ such that
   \begin{equation}
       A\in C^{h,\alpha}(\overline{U}_\eta\times[\lambda^{-1},\lambda]),\quad \Vert A\Vert_{C^{h,\alpha}(\overline{U}_\eta\times[\lambda^{-1},\lambda])}\leq E_h,\quad\textrm{and}\quad \Vert a_1-a_2\Vert_{C^{h,\alpha}(\overline{U}_\eta)}\leq E_h,
   \end{equation}
    for some $\alpha$, $0<\alpha<1$ and $\eta\leq\eta_0$, then we have
    \begin{equation}\label{eqn: der stab result}
        \Vert D^h\big(A(x,a_1(x))-A(x,a_2(x))\big)\Vert_{L^\infty(\overline{\Sigma}_\eta)}\leq C\Vert\mathcal{N}_{a_1}^\Sigma-\mathcal{N}_{a_2}^\Sigma\Vert_\ast^{\alpha\delta_h},
    \end{equation}
    where $h\geq1$ is an integer, $\delta_h=\underset{i=0}{\overset{h}{\prod}}\frac{\alpha}{\alpha+i}$, and $C$ is a positive constant which depends on the \textit{a-priori} data and $h$ only.
\end{theorem}

\section{Singular solutions having zero Neumann condition on $\partial\Omega\backslash\overline\Sigma$}\label{sec: sing solns}
We digress to construct singular solutions to \eqref{eqn: 1} having an isolated singularity at $z\in\mathbb{R}^n\setminus\overline\Omega$, of order $2-n-m$, for any integer $m\geq 0$. To do that, we consider a symmetric matrix-valued function $A=A(x)$, $x\in\Omega$ satisfying \eqref{eqn: C elliptic}. When $A\in W^{1,p} (\Omega)$, with $p>n$, such solutions were shown in \cite{Cu-G-N} to take the form for any $x\in\Omega$
 \begin{equation}\label{u global}
         u(x)=\big(A^{-1}(z_\tau)(x-z)\!\cdot\!(x-z)\big)^{\frac{2-n-m}{2}}m!\big(A^{-1}_{nn}(z)\big)^{\frac{m}{2}}
         \times C_m^{\frac{n-2}{2}}\Bigg(\frac{A^{-1}_{n}(z)(x-z)}{\big(A^{-1}_{nn}(z)\big)^{\frac{1}{2}}\big(A^{-1}(z)(x-z)\!\cdot\!(x-z)\big)^\frac{1}{2}}\Bigg)+v_0(x),
     \end{equation}
     where $C_m^{\frac{n-2}{2}}\!:\!\mathbb{C}\!\rightarrow\!\mathbb{C}$ is the complex Gegenbauer polynomial of degree $m$ and order $\frac{n-2}{2}$ and $A^{-1}_{n}(z)$, $A^{-1}_{nn}(z)$ denote the last row, last entry in the last row of the matrix $A^{-1}(z)$, respectively. Here we construct solutions to \eqref{eqn: 1} of type \eqref{u global}, having an isolated singularity in a point $z$ outside $\Omega$, and satisfying a zero Neumann condition outside $\Sigma$. We distinguish the cases when $m=0$ and $m>0$. 
\\For the case $m=0$, we construct a modified Neumann function having an isolated singularity at some point $z$ outside $\Omega$, of order $2-n$ under the mild hypotheses that $A\in C^\beta(\Omega)$, $0<\beta\leq1$ and $A$ satisfies \eqref{eqn: AR assump}, \eqref{eqn: AI assump}. We start by introducing the \textit{Neumann function} for $L$ in \eqref{eqn: 1}. From \eqref{eqn: C elliptic} we have that there exists a $2\times2$ matrix-valued function, $\mathbf{N}(x,z)$, with measurable entries $N_{ij}:\Omega\times\Omega\rightarrow\overline{\mathbb{R}}$, $i,j=1,2$, known as the Neumann matrix of $L$ in \eqref{eqn: 1} satisfying
\begin{equation}
    \begin{cases}
        \textrm{div}(\mathbf{C}(\cdot,a(\cdot))\nabla \mathbf{N}(\cdot,z))=-\delta(\cdot-z)\mathbf{I},&\textrm{in}\quad\Omega,\\
        \mathbf{C}(\cdot,a(\cdot))\nabla\mathbf{N}(\cdot,z)\cdot\boldsymbol{\nu}=-\frac{1}{|\partial\Omega|}\mathbf{I},&\textrm{on}\quad\partial\Omega,
    \end{cases}
\end{equation}
where $\boldsymbol{\nu}=(\nu,\nu)^T$ and $\mathbf{I}$ is the $2\times2$ identity matrix. Since $A\in C^\beta(\Omega)$, we also have $|\mathbf{N}(x,z)|\leq C|x-z|^{2-n}$, for all $x,z\in\Omega$, $x\neq z$ (see \cite{C-K}). We define the \textit{Neumann function} of $L$ in $\Omega$ to be $N(x,z)=N_{11}(x,z)+iN_{21}(x,z)$, where $N_{i1}$ denotes the $i$-th element in the first column of $\mathbf{N}$, for $i=1,2$. Normalising as in \cite{C-K}, $\int_{\partial\Omega}N_a(\cdot,z)\;dS(\cdot)=0$, by Green's identities, we have $N(x,z)=N(z,x)$, for all $x,z\in\Omega$, $x\neq z$.\\
For any $\eta$, $0<\eta\leq\eta_0$, we construct an enlarged domain $\Omega_\eta$, having $C^1$ boundary with constants depending only on $\eta, r_0$ such that $\Omega\subset\Omega_\eta, \quad \partial\Omega\cap\Omega_\eta\subset\subset\Sigma,\qquad\textrm{and}\qquad\textrm{dist}(x,\partial\Omega_\eta)\geq\frac{\eta}{2},\quad\textrm{for every}\quad x\in U_\eta$.
\begin{theorem}[Modified Neumann function with zero Neumann condition on $\partial\Omega\backslash\overline{\Sigma}$]\label{Neumann th}
    For any $\eta$, $0<\eta\leq\eta_0$, set $x^0\in\overline{\Sigma}_\eta$ and $z_\tau=x^0+\tau\nu$, for any $\tau$, $0<\tau\leq\tau_0$, where $\tau_0$ is a positive constant depending on $r_0$ only, such that $z_{\tau}\in\Omega_{\eta}\setminus\overline\Omega$. Let $L$ be the operator in \eqref{eqn: 1} with $A\in C^{\beta}(\Omega)$ complex-symmetric and satisfying \eqref{eqn: AR assump}, \eqref{eqn: AI assump} and $||A||_{C^{\beta}(\Omega)}\leq E$, for $0<\beta\leq 1$ and a positive constant $E$. If $S$ is an open portion of $\partial\Omega_\eta\backslash\partial\Omega$ having positive distance from $\partial\Omega$, there exists $N^{loc}\in H^1_{loc}(\overline{\Omega}_\eta\backslash \{z_\tau\})$ solution to 
    \begin{equation}\label{eqn: local NK BVP}
        \begin{cases}
            LN^{loc}(\cdot,z_\tau)=-\delta(\cdot-z_\tau),\quad&\textrm{in}\quad\Omega_\eta,\\
            A(x)\nabla N^{loc}_a(\cdot,z_\tau)\cdot\nu=0,&\textrm{on}\quad\partial\Omega_\eta\backslash S,\\
            A(x)\nabla N^{loc}(\cdot,z_\tau)\cdot\nu=-\frac{1}{|S|},&\textrm{on}\quad S,
        \end{cases}
    \end{equation}
    having the form 
    \begin{equation}\label{N asymptotics}
    N^{loc}(x,z_\tau)=C_n\big(A^{-1}(z_\tau)(x-z_\tau)\cdot(x-z_\tau)\big)^{\frac{2-n}{2}}+R(x,z_\tau),
    \end{equation}
    where $C_n$ is a suitable dimensional constant and $R(x,z_\tau)$ satisfies
    \begin{equation}\label{eqn: R property}
        |R(x,z_\tau)|+|x-z_\tau||\nabla R(x,z_\tau)|\leq C|x-z_\tau|^{2-n+\alpha},\quad\textrm{for every}\quad x\in\Omega_\eta, \quad|x-z_\tau|\leq\rho,
    \end{equation}
    where $C$ is a positive constant depending on $\mathcal{E}_1,\mathcal{E}_2,k, E$ and $\Omega_\eta$ only. Here $\rho$ is a positive number depending only on the geometry of $\Omega_\eta$, and $0<\alpha<\beta$. Moreover $\Vert N^{loc}(\cdot,z_\tau)\Vert_{H^1(\Omega)}\leq C\tau^{\frac{2-n}{2}}+B$, for any $\tau$, $0<\tau\leq\tau_0$, where $C,B$ are a positive constants depending on $\mathcal{E}_1,\mathcal{E}_2,k, E$, and $diam(\Omega)$ only.
\end{theorem}
\begin{proof}
    The Neumann kernel $N$ for $L$ in $\Omega_{\eta}$ is the distributional solution to
    \begin{equation}\label{eqn: Global NK BVP}
        \begin{cases}
            L N(\cdot,z_\tau)=-\delta(\cdot-z_\tau),\quad&\textrm{in}\quad\Omega_\eta,\\
            A(x)\nabla N(\cdot,z_\tau)\cdot\nu=-\frac{1}{|\partial\Omega_\eta|},&\textrm{on}\quad\partial\Omega_\eta.
        \end{cases}
    \end{equation}
    Considering the fundamental solution to the constant coefficients operator $\widetilde{L}=\textrm{div}(A(z_\tau)\nabla\cdot)$ in $\mathbb{R}^n$, with pole at $z_\tau$, $\Gamma(x,z_\tau)= C_n\big(A^{-1}(z_\tau)(x-z_\tau)\cdot(x-z_\tau)\big)^{\frac{2-n}{2}}$, we set $\widetilde{R}(\cdot,z_\tau)=N(\cdot,z_\tau)-\Gamma(\cdot,z_\tau)$, which is the weak solution to
    \begin{equation}\label{eqn: R BVP}
        \begin{cases}
            \textrm{div}(A(z_{\tau})\nabla \widetilde{R}(x,z_{\tau}))=\textrm{div}\big(\big(A(z_{\tau})-A(x)\big)\nabla N(x,z_{\tau})\big),\quad &\textrm{for}\quad x\in\Omega_\eta,\\
            A(z_{\tau})\nabla \widetilde{R}(x,z_{\tau})\cdot\nu=\big(A(z_{\tau})-A(x)\big)\nabla N(x,z_{\tau})\cdot\nu-A(z_\tau)\nabla\Gamma(x,z_\tau)\cdot\nu-\frac{1}{|\partial\Omega_\eta|},&\textrm{for}\quad x\in\partial \Omega_\eta.
        \end{cases}
    \end{equation}
    By Green's identities, for every $w\!\in\!\Omega_\eta$,
    \begin{align}\label{eqn: V(w,z)}
        \widetilde{R}(w,z_{\tau})=& \int_{\Omega_\eta}\big(A(z_\tau)\!-\!A(x)\big)\nabla N(x,z_{\tau})\!\cdot\!\nabla\Gamma(x,w)\;dx -\frac{1}{|\partial\Omega_\eta|}\!\int_{\partial\Omega_\eta}\Gamma(x,w)\;dS(x)\nonumber\\
        &-\int_{\partial\Omega_\eta}\Gamma(x,w)A(z_{\tau})\nabla\Gamma(x,z_{\tau})\cdot\nu\;dS(x)-\int_{\partial\Omega_\eta}R(x,z_{\tau})A(\tau)\nabla\Gamma(x,w)\!\cdot\!\nu\;dS(x),
        \end{align}
    Recalling that $|N(x,z_{\tau})|\!\leq\! C|x-z_{\tau}|^{2-n}$, for every $x\!\in\!\Omega_\eta$, by the regularity of $A$ we have $|\nabla N(x,z_{\tau})|\!\leq\! C|x-z_{\tau}|^{1-n}$, for every $x\!\in\! B_\rho(z_\tau)$, with $\rho$ sufficiently small. By taking $w\!\in\! B_\rho(z_\tau)$, the boundary integrals are uniformly bounded. For the volume integral of \eqref{eqn: V(w,z)} we have $\left\vert\int_{\Omega_\eta}\big(A(z_\tau)\!-\!A(x)\big)\nabla N(x,z_{\tau})\!\cdot\!\nabla\Gamma(x,w)\;dx\right\vert\leq |w-z_{\tau}|^{2-n+\beta}$, hence \eqref{eqn: R property} is satisfied for $x$ such that $|x-z_{\tau}|\leq \rho$ and $\alpha$ so that $0<\alpha<\beta$. By defining a ball with radius $\frac{C\tau}{2}$ centered at $z_\tau$ such that $\Omega\subset\Omega_\eta\backslash B_{\frac{C\tau}{2}}(z_\tau)$, and $\textrm{dist}(\Omega,\mathbb{R}^n\backslash\{\Omega_\eta\backslash B_{\frac{C\tau}{2}}(z_\tau)\})=(1-C)\tau$, by Caccioppoli inequality (see \cite[Theorem 4.4]{Caccioppoli}) we also obtain $\Vert N(\cdot,z_\tau)\Vert_{H^1(\Omega)}\leq \frac{C}{\tau}\Vert N(x,z_\tau)\Vert_{L^2(\Omega_\eta\backslash B_{\frac{C\tau}{2}}(z_\tau))}\leq C\tau^{\frac{2-n}{2}}$. Setting now $N^{loc}(\cdot,z_\tau)=N(\cdot,z_\tau)+v(\cdot)$, where $v$ is the solution to 
    \begin{equation}\label{eqn: BVP for w}
        \begin{cases}
            \textrm{div}(A(\cdot)\nabla v(\cdot))=0,\quad&\textrm{in}\quad\Omega_\eta,\\
            A(\cdot)\nabla v(\cdot)\cdot\nu=\frac{1}{|\partial\Omega_\eta|},&\textrm{on}\quad\partial\Omega_\eta\backslash S,\\
            A(\cdot)\nabla v(\cdot)\cdot\nu=-\frac{|\partial\Omega_\eta\backslash S|}{|\partial\Omega_\eta||S|},&\textrm{on}\quad S,
        \end{cases}
    \end{equation}
    we have that $N^{loc}$ is a solution to \eqref{eqn: local NK BVP}. Moreover, $R(\cdot,z_\tau)=\widetilde{R}(\cdot,z_\tau)+v(\cdot)$ satisfies \eqref{eqn: R property}. Again, by Caccioppoli inequality, the $H^1$-norm of $N^{loc}$ can be estimated as that of $N$ above, concluding the proof.
\end{proof}
For the case $m>0$, we  need stronger regularity assumptions on the coefficients of $L$ in \eqref{eqn: 1}. 
\begin{theorem}[Singular solutions with zero Neumann condition on $\partial\Omega\backslash\overline{\Sigma}$]\label{sing thm}
Let $x^{0}$, $z_{\tau}$ be as in Theorem \ref{Neumann th} and assume that $A=A(x)$, $x\in\Omega$ is a complex-symmetric matrix-valued function satisfying \eqref{eqn: AR assump} and \eqref{eqn: AI assump} and such that $||A||_{W^{1,p}(\Omega)}\leq E$, for some positive constant $E$. For any $m=1,2,...$, there exists $u^{loc}_m\in H_{loc}^{1}(\overline{\Omega}_\eta\backslash\{z_\tau\})\cap W_{loc}^{2,p}(\Omega_\eta\backslash\{z_\tau\})$ such that
    \begin{equation}
        \begin{cases}
            Lu^{loc}_m=0,\quad &\textrm{in}\quad\Omega_\eta\backslash\{z_\tau\},\label{eqn: sing eqn NP}\\
            A\nabla u_m^{loc}\cdot\nu=0, &\textrm{on}\quad\partial\Omega_\eta,
        \end{cases}
    \end{equation} 
    with 
    \begin{align}\label{u explicit}
         u_m^{loc}(x)\!=&\big(A^{-1}(z_\tau)(x\!-\!z_\tau)\!\cdot\!(x\!-\!z_\tau)\big)^{\frac{2-n-m}{2}}\!\!m!\big(A^{-1}_{nn}(z_\tau)\big)^{\frac{m}{2}} C_m^{\frac{n-2}{2}}\!\Bigg(\!\frac{A^{-1}_{n}(z_\tau)(x\!-\!z_\tau)}{\big(A^{-1}_{nn}(z_\tau)\big)^{\frac{1}{2}}\!\big(A^{-1}(z_\tau)(x\!-\!z_\tau)\!\cdot\!(x\!-\!z_\tau)\big)^\frac{1}{2}}\!\Bigg)\!+\!w(x).
     \end{align}
    Moreover $w$ satisfies
    \begin{eqnarray}\label{eqn: w property 1}
        |w(x)|+|x-z_\tau||Dw(x)|\leq C|x-z_\tau|^{2-n-m+\alpha},\quad \textrm{for any}\quad x\in B_\frac{\eta}{4}(z_\tau)\backslash\{z_\tau\},\\
        \label{eqn: w property 2}\Big(\int_{r<|x-z_\tau|<2r}|D^2w|^p\;dx\Big)^{\frac{1}{p}}\leq Cr^{\frac{n}{p}-n-m+\alpha},\quad\textrm{for every}\quad r,\quad0<r<\eta/8.
    \end{eqnarray}
    $\alpha$ is any number such that $0<\alpha<\beta=1-\frac{n}{p}$, and $C$ is a positive constant depending on $\alpha,n,m,p,\mathcal{E}_1,\mathcal{E}_2,k,\eta_0,\eta$ and $E$.
\end{theorem}
\begin{proof} 
    The construction of singular solutions $u$ to \eqref{eqn: 1} of type \eqref{u explicit}, having an isolated singularity at the centre of a ball $B_R(z_\tau)$ was obtained in \cite{Cu-G-N} and recalled in \eqref{u global}, where the reminder term $v_0$ satisfies \eqref{eqn: w property 1} for any $x\in B_R(z_\tau)\backslash\{z_\tau\}$, and \eqref{eqn: w property 2} for every $r$, $0<r<R/2$. Following a similar line of argument of that in the proof of Theorem \ref{Neumann th} above, we adjust $u$ in \eqref{u global} with a term $v_1$ so that the main structure of $u$ appearing in \eqref{u global} is preserved and the Neumann condition on the modified $u$ is zero outside $\Sigma$. To this end, we pick $v_1$ to be the solution to the Neumann problem 
    \begin{equation}\label{eqn: Neumann problem for v1}
        \begin{cases}
            Lv_1=0,\quad&\textrm{in}\quad\Omega_\eta,\\
            A\nabla v_1\cdot\nu=-A\nabla u\cdot\nu,&\textrm{on}\quad\partial\Omega_\eta,\\
            \int_{\Omega_\eta}v_1\;dx=0
        \end{cases}
    \end{equation}
    and set $u^{loc}_m = u+v_1$. Observing that $|\nabla u|\leq C$, on $\partial\Omega_\eta$, by the trace theorem and Poincar\'e's inequality, we have $\Vert v_1\Vert_{H^1(\Omega_\eta)}\leq C$. Moreover, by a standard interior regularity estimate \cite[Theorem 6.2.6]{Morreybook}, we also have $\Vert v_1\Vert_{W^{2,p}\big(B_{\frac{\eta}{4}}(z)\big)}\leq C$. Setting $w=v_0+v_1$, we have that $u_m^{loc}$ takes the form \eqref{u explicit} and $w$ satisfies \eqref{eqn: w property 1}, \eqref{eqn: w property 2}.
\end{proof}

We will also need the following Lemma, which proof can be found in \cite[Lemma 4.3]{mypreprint}.
\begin{lemma}\label{Grad est ND lemma}
    Let the hypothesis of Theorem \ref{sing thm} be satisfied. Then, for any $m=1,2,..$, the singular solution $u_{m}^{loc}$ having an isolated singularity at $z_\tau$ also satisfies $|Du_{m}^{loc}(x)|>C|x-z_\tau|^{1-(n+m)}$ for every $x\in\Omega_\eta$, $0<|x-z_\tau|\leq r_1$, where $C$ and $r_1$ are positive constants depending on \textit{a-priori} data.
\end{lemma}
\section{Proof of Main Results}\label{sec4}

We adapt the arguments in \cite{mypreprint} to prove Theorem \ref{stab results}. We fix an integer $h\geq 0$ and set $x^0\in\overline{\Sigma}_\eta$ such that $(-1)^h\frac{\partial^h}{\partial\widetilde{\nu}^h}(a_1-a_2)(x^0)=\Vert \frac{\partial^h}{\partial\widetilde{\nu}^h}(a_1-a_2)\Vert_{L^\infty(\overline{\Sigma}_\eta)}$ and $z_\tau=x^0+\tau\nu\in\Omega_\eta\backslash\overline{\Omega}$, for any $\tau$, $0\leq\tau\leq\min\{\tau_0,\: \frac{\eta}{8}\}$. To prove \eqref{eqn: Stab result}, we set $h=0$ and define $N_{a_i}^{loc}\in W^{2,p}(\Omega)$ to be the modified Neumann functions introduced in \eqref{eqn: local NK BVP} and corresponding to $A(\cdot,a_i(\cdot))$, 
\begin{equation}
    \label{eqn: Solns}
    N_{a_i}^{loc}(x)=C_n\big(A^{-1}(z_\tau,a_i(z_\tau))(x-z_\tau)\cdot(x-z_\tau)\big)^{\frac{2-n}{2}}+O(|x-z_\tau|^{2-n+\alpha}),
\end{equation} 
for $i=1,2$. By fixing $\rho>0$ and possibly reducing $\tau$, such that $0<\tau\leq\min\{\tau_0,\: \frac{\eta}{8},\:\frac{\rho}{2}\}$, we have that $B_\rho(z_\tau)\cap\Omega\neq\emptyset$ and $B_\rho(z_\tau)\cap\Omega\subset U_{\eta}$. Feeding Alessandrini's identity \eqref{eqn: Aless for ND} with $ N_{a_i}^{loc}$ in \eqref{eqn: Solns}, we obtain 
\begin{align}
    \label{eqn: I1}
\bigg|\int_{B_\rho(z_\tau)\cap\Omega}&\!\frac{F(x)}{\big|A^{-1}(z_\tau, a_1(z_\tau))(x-z_\tau)\cdot(x-z_\tau)\big|^n\big|A^{-1}(z_\tau, a_2(z_\tau))(x-z_\tau)\cdot(x-z_\tau)\big|^n}\;dx\bigg|\nonumber\\
    \leq&C\biggl\{\!\int_{\Omega\backslash B_\rho(z_\tau)}\!\!\!\!\!\!\!\!\!\!\!\!\!\!\!\!|x-z_\tau|^{2-2n}\;dx+\!\int_{B_\rho(z_\tau)\cap\Omega}\!\!\!\!\!\!\!\!\!\!\!\!\!\!\!\!|x-z_\tau|^{2-2n+\beta}\;dx+\!\int_{B_\rho(z_\tau)\cap\Omega}\!\!\!\!\!\!\!\!\!\!\!\!\!\!\!\!|x-z_\tau|^{2-2n+\alpha}\;dx+\!\int_{B_\rho(z_\tau)\cap\Omega}\!\!\!\!\!\!\!\!\!\!\!\!\!\!\!\!|x-x^0|^\beta|x-z_\tau|^{2-2n}\;dx\!\biggr\}\nonumber\\
    &+\Vert\mathcal{N}_{a_1}^\Sigma-\mathcal{N}_{a_2}^\Sigma\Vert_{\ast}\Vert A(x,a_1(x))\nabla N_{a_1}^{loc}(x)\cdot\nu\Vert_{\prescript{}{0}{H}^{-1/2}(\Sigma)}\Vert A(x,a_2(x))\nabla N_{a_2}^{loc}(x)\cdot\nu\Vert_{\prescript{}{0}{H}^{-1/2}(\Sigma)},
\end{align}
where $0<\alpha<\beta=1-\frac{n}{p}$ and $F$ is defined by $F(x)=\big(A^{-1}(x^0,a_2(x^0))-A^{-1}(x^0,a_1(x^0))\big)(x-z_\tau)\!\cdot\!(x-z_\tau)
    \times\big(\overline{A^{-1}(z_\tau, a_1(z_\tau))}(x-z_\tau)\!\cdot\!(x-z_\tau)\big)^{\frac{n}{2}}\!\big(\overline{A^{-1}(z_\tau, a_2(z_\tau))}(x-z_\tau)\!\cdot\!(x-z_\tau)\big)^{\frac{n}{2}}$.
The choice of $k$ in \eqref{eqn: k1} implies $|\Im F(x)|\leq|\Re F(x)|$ and $\Re F(x)>0$. Therefore, by \eqref{eqn: mono}, we compute $|F(x)|\geq C\big(a_1(x^0)-a_2(x^0)\big)|x-z_\tau|^{2+2n}$.
Now estimating all the integrals and the $\prescript{}{0}{H}^{-1/2}(\Sigma)$-norms in \eqref{eqn: I1}, we obtain $
    \Vert a_1-a_2\Vert_{L^\infty(\overline{\Sigma}_\eta)}\leq C\big\{C\tau^\beta+C\tau^\alpha+C\tau^{n-2}+ \Vert\mathcal{N}_{a_1}^\Sigma-\mathcal{N}_{a_2}^\Sigma\Vert_{\ast}\big\}$,
and, by taking $\tau\rightarrow0$, we obtain
\begin{eqnarray}\label{eqn: induction j=0 final}
    \Vert a_1-a_2\Vert_{L^\infty(\overline{\Sigma}_\eta)}\leq C\Vert\mathcal{N}_{a_1}^\Sigma-\mathcal{N}_{a_2}^\Sigma\Vert_{\ast},
\end{eqnarray}
which, by the regularity of $A$, leads to \eqref{eqn: Stab result}. To prove \eqref{eqn: der stab result}, by induction on $j$, we show that for $j\leq h$, 
\begin{equation}
    \label{eqn: induction aim}
    \Big\Vert \frac{\partial^j}{\partial\widetilde{\nu}^j}(a_1-a_2)\Big\Vert_{L^\infty(\overline{\Sigma}_\eta)}\leq C\Vert\mathcal{N}_{a_1}^\Sigma-\mathcal{N}_{a_2}^\Sigma\Vert_\ast^{\delta_j}.
\end{equation} 
For $j=0$, \eqref{eqn: induction aim} is given by \eqref{eqn: induction j=0 final}. For the inductive step, assuming that \eqref{eqn: induction aim} holds true for any $j$, $j\leq h$, one can show that it holds true for $j=h$ too by feeding Alessandrini's identity \eqref{eqn: Aless for ND}, for a fixed $m>0$, with the singular solution $u_{m;a_i}^{loc}\in W^{2,p}(\Omega)$ constructed in Theorem \ref{sing thm}, having a singularity at $z=z_\tau$ and corresponding to $A(\cdot,a_i(\cdot))$, for $i=1,2$.
Following the same line of reasoning of \cite[Proof of Theorem 7]{mypreprint} and by Lemma \ref{Grad est ND lemma} we have
\begin{flalign}\label{Alessandrini derivatives}
    \Vert\mathcal{N}_{a_1}^\Sigma-\mathcal{N}_{a_2}^\Sigma\Vert_{\ast}\Vert A(x,a_1(x))\nabla u_{m; a_1}^{loc}(x)\cdot\nu&\Vert_{\prescript{}{0}{H}^{-1/2}(\Sigma)}\Vert A(x,a_2(x))\nabla u_{m: a_2}^{loc}(x)\cdot\nu\Vert_{\prescript{}{0}{H}^{-1/2}(\Sigma)}\nonumber\\
    \geq& C\int_{B_\rho(z_\tau)\cap\Omega}\big(a_1(x)-a_2(x)\big)|x-z_\tau|^{2-2n-2m}\,dx-C,
\end{flalign}
leading to $\Vert\frac{\partial^h}{\partial\widetilde{\nu}^h}(a_1-a_2)\Vert_{L^\infty(\overline{\Sigma}_\eta)}\leq C\big\{\Vert\mathcal{N}_{a_1}^\Sigma-\mathcal{N}_{a_2}^\Sigma\Vert_\ast^{\delta_{h-1}}\tau^{-h}+\tau^\alpha+\tau^{n+2m-2-h}\big\}$ via estimation of the integrals and the $\prescript{}{0}{H}^{-1/2}$-norms in \eqref{Alessandrini derivatives}.
By optimising the latter inequality with respect to $\tau$ and choosing $m$ sufficiently large and, we obtain $\Vert D^h(a_1-a_2)\Vert_{L^\infty(\overline{\Sigma}_\eta)}\leq C\big\{\Vert\mathcal{N}_{a_1}^\Sigma-\mathcal{N}_{a_2}^\Sigma\Vert_\ast^{\delta_{h}}$,
hence, arguing as above, from the latter we obtain \eqref{eqn: induction aim}. By iteratively using the interpolation inequality in \cite[(127)]{mypreprint}, we obtain \eqref{eqn: der stab result}, concluding the proof.\qed
\bibliographystyle{plain}
\bibliography{references}
\end{document}